\documentclass[12pt,reqno]{amsart}
\usepackage{a4wide}

\usepackage{amsmath}
\usepackage{amssymb}
\usepackage{amsfonts}
\usepackage{color}

\usepackage{graphicx}
\usepackage{subfig}
\usepackage{bm}
\usepackage{tikz}

\newtheorem{lemma}{Lemma}[section]

\newtheorem{theorem}[lemma]{Theorem}

\theoremstyle{definition}

\theoremstyle{remark}
\newtheorem{remark}[lemma]{Remark}

\DeclareMathOperator{\supp}{supp}

\newcommand{\dst}{\displaystyle}

\def\C{{\mathbb{C}}}

\renewcommand{\H}{{\mathrm{H}}}
\def\K{{\mathrm{K}}}

\def\N{{\mathbb{N}}}

\def\R{{\mathbb{R}}}

\def\Z{{\mathbb{Z}}}

\def\cc{{\mathcal C}}

\newcommand{\norm}[1]{{\left\|{#1}\right\|}}

\newcommand{\scal}[1]{{\left\langle{#1}\right\rangle}}

\DeclareMathOperator{\sinc}{sinc}

\newcommand{\ud}{\,\mathrm{d}}
\newcommand{\vep}{\varepsilon}
\renewcommand{\phi}{\varphi}

\numberwithin{equation}{section}
\begin{document}
\sloppy

\title[Oversampling and Donoho-Logan type theorems in model spaces]{Oversampling and Donoho-Logan type theorems in model spaces}

\author[A. Baranov]{Anton Baranov}
\address{St. Petersburg State University, Department of Mathematics and Mechanics, 28, Universitetskii prosp., 
198504 Staryi Petergof, Russia.}
\email{anton.d.baranov@gmail.com}

\author[P. Jaming]{Philippe Jaming}
\address{Univ. Bordeaux, CNRS, Bordeaux INP, IMB, UMR 5251, F-33400 Talence, France}

\email{philippe.jaming@math.u-bordeaux.fr}
\author[K. Kellay]{Karim Kellay}
\address{Univ. Bordeaux, CNRS, Bordeaux INP, IMB, UMR 5251, F-33400 Talence, France}

\email{kkellay@math.u-bordeaux.fr}
\author[M. Speckbacher]{Michael Speckbacher}

\address{Universit\"at Wien, Fakult\"at f\"ur Mathematik, Oskar-Morgenstern-Platz 1, 1090 Wien, Austria.}
\email{michael.speckbacher@univie.ac.at}


\begin{abstract}
The aim of this paper is to extend two results from the Paley--Wiener setting to more general
model spaces. The first one is an analogue of the oversampling Shannon sampling formula. The second one
is a version of the Donoho--Logan Large Sieve Theorem which is a quantitative estimate of the embedding of
the Paley--Wiener space into an $L^2(\R,\mu)$ space.
\end{abstract}

\date{}
\maketitle


\section{Introduction}

The aim of this paper is to extend two classical results on the Paley--Wiener space
to more general model spaces. The first result is the sampling theorem, or more
precisely the oversampling formula that improves the convergence in Shannon sampling.
The second result is the Donoho--Logan Large Sieve Principle which can be seen as one of the
first results on sparsity in signal processing. In the context
of complex analysis, this result is a result on Carleson measures for
the Paley--Wiener space.

Let us now be more precise. Recall that the Paley--Wiener space is the subspace of $L^2$ signals that are band-limited to $[-c,c]$. It is a very common space used to model signal encountered in natural sciences. 
If we normalize the Fourier transform through
$$
\widehat{f}(\xi)=\int_{\R}f(t)e^{-i t\xi}\,\mbox{d}t,
$$ 
the Paley--Wiener space $PW_c^p(\R)$, $c>0$, $1\le p <\infty$, is defined as
$$
PW_c^p(\R)=\Big\{f\in L^p(\R)\,:\ \supp\widehat{f}\subset[-c,c]\Big\},
$$
where, for $p>2$, $\widehat{f}$ is understood in the distributional sense.
Two well known properties of the Paley--Wiener spaces are that they consist of entire functions and that
 every function $f\in PW_c^2(\R)$ can be reconstructed
from  samples $\{f(\pi k/b)\}_{k\in\Z}$ via the Kotelnikov--Nyquist--Shannon Formula
\begin{equation}
\label{eq:shannon1}
f(x)=\sum_{k\in\Z}f\left(k\frac{\pi}{b}\right)\frac{1}{2b}\widehat{\gamma}\left(x-k\frac{\pi}{b}\right),
\end{equation}
where $b\geq c$ and $\gamma$ is any even function supported in $[-b,b]$ with $\gamma(\xi)=1$
for $\xi\in[-c,c]$. Taking $b\geq c$ and $\gamma=\mathbf{1}_{[-b,b]}$
we get the classical Shannon Formula
\begin{equation}
\label{eq:shannon2}
f(x)=\sum_{k\in\Z}f\left(k\frac{\pi}{b}\right)\sinc \left(b\left(x-k\frac{\pi}{b}\right)\right),
\end{equation}
where $\sinc t={\sin t}/{t}$. However, when $b>c$ one can do better by taking $\gamma$
smooth in which case $\widehat{\gamma}$ decreases faster than the sinc function. Most authors take $\gamma$ to be 
$\cc^\infty$ so that, for every $N$, $|\widehat{\gamma}\left(x-k\frac{\pi}{b}\right)|=O(k^{-N})$ when $k\to\pm\infty$
and this estimate is even uniform when $x$ stays in a compact set. As a consequence \eqref{eq:shannon1}
converges much better than \eqref{eq:shannon2}, a fact that is well known in signal processing.
The drawback of this is that, to the best of our knowledge, there is no example of a function $\gamma$
as above for which $\widehat{\gamma}$ is explicitly known.

One way to overcome this is to give up on arbitrarily fast decay and to fix $N$. One then fixes a parameter $a>0$,
and takes $\psi^{(1)}=\psi=\frac{1}{2a}\mathbf{1}_{[-a,a]}$ and $\psi^{(k+1)}=\psi^{(k)}*\psi$.
Then $\psi^{(N)}$ is supported in $[-Na,Na]$ and $\int_\R\psi^{(N)}=1$.
It follows that $\gamma=\psi^{(N)}*\mathbf{1}_{[-c-Na,c+Na]}$ is even, supported in $[-b,b]$ with $b=c+2Na$
and that $\gamma(\xi)=\int_\R\psi^{(N)}=1$ for $\xi\in[-c,c]$.
Computing $\widehat{\gamma}$, we get
\begin{equation}
\label{eq:shannon3}
f(x)=2(c+Na)\sum_{k\in\Z}f\left(k\frac{\pi}{b}\right)\left[\sinc\left( a\left(x-k\frac{\pi}{b}\right)\right)\right]^N\sinc \left((c+Na)\left(x-k\frac{\pi}{b}\right)\right).
\end{equation}

The second result we are dealing with in this paper is the Donoho--Logan Large Sieve theorem.
The analytic large sieve principle is a classical inequality for trigonometric polynomials which finds many 
applications in analytic number theory (see, {\it e.g.}, \cite{mont78} and references therein). It was extended
from trigonometric polynomials to their near cousins the band-limited functions by Donoho and Logan \cite{dolo92}
in the early 90s (after earlier work by Boas \cite{Bo}), applying it to reconstruction of missing data in signal processing.

For $I$ an interval, we denote by $|I|$ the length of $I$. When
$\mu$ is a positive $\sigma$-finite measure and $\delta>0$, we write
\begin{equation}
\label{eq:dmudelta}
D_\mu(\delta)=\sup\left\{\frac{\mu(I)}{|I|}\,:\ I\mbox{ closed interval, }|I|=\delta\right\}
=\sup_{x\in\R}\frac{\mu([x,x+\delta])}{\delta}.
\end{equation}
Donoho and Logan proved the following:

\begin{theorem}[Donoho--Logan]
Let $c,\delta>0$ and $\mu$ be a positive $\sigma$-finite measure. 
Then for every $f\in PW_c^2(\R)$,
\begin{equation}
\label{eq:ds}
\int_{\R}|f(x)|^2\,\ud\mu(x)\leq\left(1 + \frac{c\delta}{\pi}\right) 
D_\mu(\delta)\int_{\R}|f(x)|^2\,\ud x.
\end{equation}
Moreover, for every $f\in PW_c^1(\R)$,
\begin{equation}
\label{eq:ds2}
\int_{\R}|f(x)|\,\ud\mu(x)\leq \frac{D_\mu(\delta)}{\sinc\left(\dfrac{c\delta}{2}\right)}
\int_{\R}|f(x)|\,\ud x.
\end{equation}
\end{theorem}

This theorem is one of the founding results of the use of sparsity constraints in signal recovery problems and has led to considerable research in the decades following it. However, until recently, it seems that this result 
did not get the attention it deserves in other mathematical communities.
For instance, in complex analysis, the Donoho--Logan Large Sieve theorem 
provides a quantitative estimate of the fact that a
positive measure is a Carleson 
measure of the Paley--Wiener space, that is, a bound on the norm of the injection $PW_c^2(\R)\hookrightarrow L^2(\mu)$.
Note that a first estimate of this type was given by Lin \cite{L}. Recently, Husain and Littmann extended the result for $p=1$ to higher dimensions \cite{huli}.

This inadvertence may be due to the fact that the strategy of proof
has some rigidities that are difficult to overcome. Indeed, the proof relies heavily on
the interplay between convolution and the Fourier transform, and on a construction by Selberg (based on a
previous construction by Beurling) of an entire function majorizing the sign function, see \cite{mont78,Va}.
Despite its rigidity, this strategy of proof has recently been extended to more general setting,
see, {\it e.g.}, \cite{abspe18-sieve,abspe22,jasp1}.
Our first aim in this paper is to provide two new proof strategies,
one based on the oversampling formula, the second on Bernstein's inequality.
In particular, those proofs rely on real variable techniques only.
The price to be paid is that the numerical constants we obtain are slightly worse.
The proofs however offer more flexibility so that they apply in larger settings.

To illustrate this, we will extend the oversampling formula and the large sieve inequality
to the setting of model spaces on the upper half-plane,
an important family of spaces of holomorphic functions which contains the
Paley--Wiener space $PW_c^2(\R)$. Let us now describe this family.
First, let $\H^2$ be  the Hardy space on the upper half-plane  $\C^+:=\{z\in\C:\ \text{Im}\,z>0\}$. 
$$
\H^2:=\Big\{f\in \text{\rm Hol}(\C^+)\,:\ \sup_{ y>0}\int_{\R}|f(x+iy)|^2\mbox{d}x<\infty\Big\}.
$$
Note that we may identify $\H^2=\{f\in L^2(\R)\,:\ \supp \widehat{f}\subset[0,+\infty)\}$.
Let $\Theta$  be an inner function in $\C^+$,  that is, a bounded analytic function on $\C^+$ with unimodular boundary values almost everywhere on $\R$. 
The corresponding \emph{model space} is defined by
$$
\K_\Theta^2:=\H^2\cap(\Theta\H^2)^\bot.
$$
As a fundamental example, if $\Theta_{c}(z)= \exp(icz)$ for some $c>0$, then
$$
\K_{\Theta_{2c}}^2=\{f\in L^2(\R)\,:\ \supp \widehat{f}\subset[0,2c]\}
=\{\Theta_c f\,: f\in PW_c^2(\R)\}
$$
or, equivalently,
$$
PW_c^2(\R)=\{e^{-icz}f\,:\ f\in\K_{\Theta_{2c}}^2\}.
$$
Moreover, one can   define $\K_\Theta^p$ model spaces in $H^p$, $1\leq p\leq+\infty$, so as
to also cover the $PW_c^p$ spaces, {\it see} Section \ref{Sec:5}.
In the context of model spaces, an analogue of Shannon's sampling formula was established by de Branges \cite{de}
in the case of meromorphic inner functions and later by Clark \cite{C} in the general case, 
see \eqref{eq:clarkshan}. However, the oversampling formula for model spaces seems to be unknown.
Previously,   oversampling  results were obtained in \cite{STU,TU}  for specific
model spaces (or closely related de Branges spaces) associated with  certain differential operators. 

We start by establishing an oversampling formula for the  class of model spaces with meromorphic inner functions. This condition  is equivalent to
$\Theta'\in L^\infty(\R)$.
Once this is done, we propose a first proof of Donoho and Logan's theorem 
for model spaces associated with inner functions with bounded derivative,
an interesting class of model spaces that share many properties with the Paley--Wiener spaces (see,
{\it e.g.},  \cite{dyak,Ni}). 

\begin{theorem}
\label{th:intro2}
Let $\Theta$ be an  inner function such that $\Theta'\in L^\infty(\R)$, and $\mu$
be a $\sigma$-finite Borel measure on $\R$. For $\delta>0$, $D_\mu(\delta)$ defined in \eqref{eq:dmudelta}
and
$1\leq p<+\infty$ there exists $C_p>0$ such that  for every $f\in \K^p_\Theta$
\begin{equation}
\label{1}
\int_\R |f(x)|^p\,\mathrm{d}\mu(x) \leq 
 C_p\left(1+\|\Theta'\|_\infty\delta\right)^pD_\mu(\delta)
 \int_\R |f(x)|^p\,\mathrm{d}x.
\end{equation}
\end{theorem}
 
The  proof relying on oversampling gives this result for $p=2$ and a rather bad 
constant $C$ but gives the right behavior with respect to $\delta$
and $\|\Theta'\|_\infty$. We thus give a second proof, based on Bernstein's inequality which allows
to obtain the constant $C_p=1$, falling only short of Donoho and Logan's result in the
case of the Paley--Wiener space. The particular constants are discussed in more detail in  Remark~\ref{remark}. We think that both proofs are interesting due to the fact
that they only rely on real analytic arguments and are rather flexible. 
Finally, we end this article with a version of Theorem~\ref{th:intro2}   that applies to a {different}
class of inner functions (not necessarily meromorphic), the so-called one-component inner functions, but requires a modification in the definition of $D_\mu(\delta)$, {\it see} Theorem~\ref{th:onecomp}. This provides a quantitative converse version of the
Logvinenko--Sereda Theorem for model spaces established in \cite{HJK}.

\medskip

The rest of this paper is organized as follows:
Section~\ref{Sec:2} is devoted to a technical lemma. In Section~\ref{Sec:3}, we introduce
the necessary background on model spaces and prove the oversampling formula. 
In Section~\ref{Sec:4} we give an extension of the large sieve inequality
to model spaces using the oversampling techniques. The approach based on Bernstein type inequalities 
is considered in Section~\ref{Sec:5}, while
in Section~\ref{Sec:6} a certain analog of the large sieve is given for a class of model spaces generated by the so-called
one-component inner functions. 
\bigskip


\section{Preliminary lemma}\label{Sec:2}
Let $\Xi$ be a continuous function such that
\begin{equation}
\label{eq:propxi}
|\Xi(x)|\leq \min(1,|x|^{-1}).
\end{equation}
As an example, one can take $\Xi(x)=\sinc x:=\dfrac{\sin x}{x}$.

\begin{lemma}\label{schur} 
Let $\Xi$ be a function satisfying the bound \eqref{eq:propxi}.
Let $a,b\in \R$, $\alpha,\delta>0$  and $\mu$
be a $\sigma$-finite Borel measure on $\R$.  Then we have
\begin{enumerate}
\item $\displaystyle \int_\R \Xi(x-a)^2\Xi( x-b)^2\,\mathrm{d}x\leq \frac{8\pi}{4+(a-b)^2}$,
\item $\displaystyle \int_\R \Xi(x-a)^2\Xi(x-b)^2\,\mathrm{d}\mu(x)\leq 8\pi \frac{C_\delta^2}{\delta} \frac{\sup_{x\in \R} \mu([x,x+\delta])}{4+(b-a)^2}$,
with $C_\delta=\max(4,1+9\delta^2)$.
\end{enumerate}
\end{lemma}

\begin{proof}
The first estimate follows immediately from the simple estimate
$$
|\Xi(x)|^2\leq \min(1,|x|^{-2}) \leq \frac{2}{1+x^2},
$$
and the computation of the integral
$$
\displaystyle \int_\R \frac{\mathrm{d}t}{(1+(t-a)^2)(1+(t-b)^2)}=  \frac{2\pi}{4+(a-b)^2}, 
$$
via the residue formula.

Next, for $\ell\in\Z$, let  $I_\ell=[\ell \delta, (\ell+1)\delta[$, we have 
\begin{align*}
 \int_\R \Xi(x-a)^2\Xi(x-b)^2\,\mathrm{d}\mu(x)
 &=\sum_{\ell\in \Z} \int_{I_\ell} \Xi(x-a)^2\Xi(x-b)^2\,\mathrm{d}\mu(x)
 \\
 &\leq \sup_{\ell\in \Z} \mu(I_\ell) \cdot \sum_{\ell\in \Z} \frac{4}{(1+d(a,I_\ell)^2)(1+d(b,I_\ell)^2)}\\
&=\frac{4}{\delta}\sup_{\ell\in \Z} \mu(I_\ell) \cdot \sum_{\ell\in \Z}
\int_{I_\ell}\frac{\mbox{d}x}{\big(1+d(a,I_\ell)^2\big)\big(1+d(b,I_\ell)^2\big)}.
\end{align*}
Now, let $x\in I_\ell$. If $a\leq (\ell-2)\delta$ then $x-a\geq 2\delta$ so that
$$
d(a,I_\ell)=\ell\delta-a=(\ell+1)\delta-a-\delta\geq x-a-\delta\geq \frac{x-a}{2}
$$
thus
$$
\frac{1}{1+d(a,I_\ell)^2}\leq \frac{4}{1+(x-a)^2}.
$$
The same holds if $a\geq (\ell+3)\delta$. Finally, if $(\ell-2)\delta\leq a\leq (\ell+3)\delta$,
$x-a\leq 3\delta$ thus 
$$
\frac{1}{1+d(a,I_\ell)^2}\leq 1\leq \frac{1+9\delta^2}{1+(x-a)^2}.
$$
It follows that
\begin{align*}
 \int_\R \Xi(x-a)^2\Xi(x-b)^2\,\mathrm{d}\mu(x)
& \leq \frac{4}{\delta} C_\delta^2\sup_{\ell\in \Z} \mu(I_\ell) \cdot  \int_\R \frac{\mbox{d}x}{(1+(x-a)^2)(1+(x-b)^2)}\\
&\leq 8\pi \frac{C_\delta^2}{\delta}\, \frac{\sup_{x\in \R} \mu\big([x,x+\delta]\big)}{4+(b-a)^2},
\end{align*}
with $C_\delta=\max(4,1+9\delta^2)$.
\end{proof}


Applying Lemma \ref{schur}(2) to the dilated measure $\mu_\alpha(A)=\mu(A/\alpha)$, $\alpha>0$
we get
\begin{equation}
\label{eq:schur}
\int_\R \Xi(\alpha x-a)^2\Xi(\alpha x-b)^2\,\mathrm{d}\mu(x)\leq 
8\pi \frac{C_\delta^2}{\delta}\frac{\sup_{x\in \R} \mu([x,x+\delta/\alpha])}{4+(b-a)^2},
\end{equation}
with $C_{\delta}=\max(4,1+9\delta^2)$.

\medskip

A slightly more evolved version of the lemma is easily available: for $m\geq 2$ and $\delta>0$
there is a constant $C_{m,\delta}$ such that
\begin{equation}
\label{eq:highertech}
\int_\R \Xi(x-a)^{2m}\Xi(x-b)^{2m}\,\mathrm{d}\mu(x)\leq  
\frac{C_{m,\delta}}{\bigl(1+(b-a)^2\bigr)^m}\sup_{x\in \R} \frac{\mu([x,x+\delta])}{\delta}.
\end{equation}
This follows from the easily established inequality
$$
\int_\R |\Xi(x-a)|^{2m}|\Xi( x-b)|^{2m}\,\mathrm{d}x\leq \sqrt{\pi}2^{2m+1}\frac{\Gamma(m-1/2)}{\Gamma(m)}\frac{1}{\bigl(1+(b-a)^2\bigr)^{m}}.
$$
%


\section{Oversampling and the Large Sieve in model spaces}\label{Sec:3}

\subsection{Background on  model spaces} The Hardy space on the upper-half plane 
$\C^+:=\{z\in\C:\ \text{Im}\,z>0\}$,  $\H^2=\H^2(\C^+)$ contains all holomorphic functions on $\C^+$ for which
$$
\sup_{y>0}\int_\R |f(x+iy)|^2\,\mbox{d}x<\infty.
$$
Every function  $f\in\H^2$ has an almost everywhere defined ``vertical'' boundary function $f(x):=\lim_{y\rightarrow 0}f(x+iy)$, and $f\in L^2(\R)$ which may be used to define an inner product on $\H^2$ 
$$
\langle f,g\rangle :=\int_\R f(x)\overline{g(x)}\,\mbox{d}x.
$$
We say that an analytic function $\Theta$ on $\C^+$ is \emph{inner} if $|\Theta|\leq1$ on $\C^+$ and if the almost everywhere defined boundary function $\Theta(x)$, $x\in\R$, has modulus one. 
If $\Theta$ is an inner function, then the corresponding \emph{model space} is defined by
\begin{equation}\label{def:model-space}
\K^2_\Theta:=\H^2\ominus\Theta\H^2=(\Theta\H^2)^\bot.
\end{equation}
Recall that the reproducing kernel for  functions $\K^2_\Theta$ is defined by 
\begin{equation}\label{eq:RK}
k^\Theta_z(w)=\frac{i}{2\pi}\frac{1-\overline{\Theta(z)}\Theta(w)}{w-\overline{z}},\qquad  z,w\in\C^+.
\end{equation}
For every $f\in\K^2_\Theta$, 
$$
f(z)=\langle f,k^\Theta_z\rangle,
$$
and 
$$
k^\Theta_z(z) =\|k^\Theta_z\|^2=\frac{1-|\Theta(z)|^2}{4\pi\text{Im}\,z}.
$$ 
Each inner function can be factored as 
$$
\Theta(z)=e^{i\tau}\Theta_c(z)B_\lambda (z)S_\mu(z),\qquad z\in \C^+,
$$
where $\tau$ is a real constant, $\Theta_c(z)={e^{ icz}}$, $c\geq0$,
$$
B_\Lambda(z)= \prod_{\lambda\in \Lambda}  e^{i\alpha_\lambda} \frac{z-\lambda}{z-\overline{\lambda}}, 
$$
is the {\em Blaschke product} with zeros $\lambda\in \Lambda\subset\C^+$, repeated according to multiplicity,  
satisfying the Blaschke condition
\begin{equation}
\label{cond:blaschke}
\sum_{\lambda\in \Lambda}\frac{\textrm{Im}\, \lambda}{1+|\lambda|^2}<\infty,
\end{equation}
and $\alpha_\lambda\in \R$, 
$$
S_\mu(z)=\exp \left(i\int_\R \left(\frac{1}{x-z}-\frac{t}{x^2+1}\right)\mbox{d}\mu(x)\right),
$$
where $\mu$ is a singular measure with respect to the Lebesgue measure  such that 
$$
\int_\R \frac{\mbox{d}\mu(x)}{1+x^2}<\infty.
$$
The spectrum of $\Theta$ is the closed set 
\begin{equation}
\label{spec}
\rho(\Theta):=\big\{\zeta\in \overline{\C^+} \cup\infty\text{ : }
\liminf_{\stackrel{z\to\zeta}{z\in \C^+}}|\Theta(z)|=0\big\}.
\end{equation}
Note that $\Theta$, along with every function in $\K^2_\Theta$, 
has an analytic extension across any interval of  $\R\setminus \rho(\Theta)$. 

By the Ahern--Clark theorem \cite{AC}, $k^\Theta_x\in\K^2_\Theta$, $x\in\R$, 
if and only if the modulus of the angular derivative of $\Theta$ is finite. This means that
$$
|\Theta'(x)|=a+ 2\sum_{\lambda\in \Lambda}\frac{\textrm{Im}\, \lambda}{|x-\lambda|^2}+\int_\R \frac{\mbox{d}\mu(t)}{(t-x)^2}<\infty.
$$
For any $\alpha \in \C$ , $|\alpha|=1$ the function $(\alpha+\Theta)/(\alpha-\Theta)$ has a positive real part in the upper half plane, and  hence, by the Herglotz--Riesz representation theorem, there exist $c_\alpha>0$ and a 
non-negative measure $\sigma_{\Theta}^{\alpha}$, called  the Clark measure, such that
$$
\textrm{Re}\,  \frac{ \alpha+\Theta(z)}{\alpha-\Theta(z)}= 
c_\alpha \textrm{Im}\, z + \frac{\textrm{Im}\, z}{\pi} \int_{\R} \frac{\mbox{d}\sigma_{\Theta}^{\alpha}(x)}{|x-z|^2},\qquad  z\in \C^+.
$$
The Clark measure $\sigma_{\Theta}^{\alpha}$ is carried by the set $\{x\in \R\text{ : } \lim_{ y\to 0+}\Theta(x+iy)=\alpha\}$. 

Clark \cite{C} showed that if $c_\alpha = 0$ and $\sigma_{\Theta}^{\alpha}$  is purely atomic, that is 
$$
\sigma_{\Theta}^{\alpha}=\sum a_n\delta_{x_n}, 
$$
where $\delta_x$  denotes the Dirac measure at the point $x$, then the system $\{k_{x_n}^{\Theta} \}$ is an orthogonal basis in $K_{\Theta}^{2}$.
In particular, one has
\begin{equation}\label{eq:clarkshan}
f(z)=\sum_{n\in\Z}f(x_n)\frac{k^\Theta_{x_n}(z)}{\|k^\Theta_{x_n}\|^2},\qquad f\in \K^2_\Theta,
\end{equation}
and 
\begin{equation}\label{eq:plancherel}
\|f\|^2=\sum_{n\in\Z}\frac{|f(x_n)|^2}{\|k^\Theta_{x_n}\|^2}.
\end{equation}
We end this section with a discussion of a special class of inner functions
which will be considered throughout the paper (with exception of Section \ref{Sec:6}).
A {\it meromorphic inner function} on  $\C^+$ 
is an inner function on $\C^+$  with a meromorphic continuation to $\C$. Any 
meromorphic inner function $\Theta$ can be represented as
$$
\Theta=  \Theta_cB_\Lambda,
$$
where $c\geq 0$ and $B_\Lambda$ is the Blaschke product associated
with the zeroes $\Lambda=\{\lambda_n\}$ of $\Theta$ which satisfy
$|\lambda_n|\to \infty$ as well as the Blaschke condition \eqref{cond:blaschke}.
All elements of the corresponding model space $\K^2_\Theta$ are also meromorphic in $\C$, 
and there is a canonical isomorphism of such model spaces with de Branges' Hilbert spaces
of entire functions  \cite{de}.

By  the Riesz--Smirnov factorization there exists an increasing, real analytic function $\varphi:\R\rightarrow\R$ such that
$$
\Theta(x)=\text{exp}(i\varphi(x)).
$$
In that case $|\Theta'|=\varphi'$ and 
\begin{equation}\label{eq:norm-Kt}
\|k^\Theta_x\|^2=\frac{\varphi'(x)}{2\pi}.
\end{equation}
Then, by the Cauchy--Schwarz inequality one has
\begin{equation}\label{eq:pointwise-Kx}
|k^\Theta_x(t)|=|\langle k^\Theta_x,k^\Theta_t\rangle|\leq \frac{\sqrt{\varphi'(x)\varphi'(t)}}{2\pi},\qquad x,t\in\R.
\end{equation}

In the case of meromorphic inner functions the Clark measure construction becomes much more transparent. 
In this setting such measures were introduced by de Branges (see, {\it e.g.}, \cite{de}) long before the work of Clark. 
For model spaces associated with meromorphic inner functions orthogonal bases of reproducing kernels 
can be constructed as follows. 
For $\gamma\in [0,2\pi)$ define the set of points $\{x_n\}_{n\in \Z}$ by
\begin{equation}\label{eq:def-tn}
\varphi(x_n)=\gamma+2\pi n,\qquad n\in\Z.
\end{equation}
(Note that the points $x_n$ may not exist for all $n\in\Z$.) Then the family of normalized reproducing kernels 
$\{k^\Theta_{x_n}/ \|k^\Theta_{x_n}\| \}_{n\in\Z}$, with the points $\{x_n\}_{n\in\Z}\in \R$  given by \eqref{eq:def-tn}, 
forms an orthonormal  basis for $\K^2_\Theta$ for each $\gamma\in[0,2\pi)$, except, maybe, one 
(in the case that $\Theta-e^{i\gamma}\in L^2(\R)$).  

In what follows we will consider the class of inner functions such that 
$\Theta' \in H^\infty(\C^+)$. This condition implies that $\Theta$ is meromorphic and is equivalent to
$\Theta' =\phi' \in L^\infty(\R)$. It was noticed already in \cite{dyak} that the model spaces $\K_\Theta^2$ 
with  $\Theta' \in L^\infty(\R)$ have many properties analogous to the properties of the Paley--Wiener 
spaces $PW_c^2(\R)$.
\bigskip


\subsection{Enlarging model spaces and oversampling}

We need  the following classical result \cite{Ni}. However, for the sake of completeness, we give the complete proof here.

\begin{lemma}\label{cor:enlarge-theta}
Let $\Theta$ be an inner function, let  $\Theta_1=e^{i\tau_1}\Theta_{c_1}B_{\Lambda_1}$, and $\Theta_2=e^{i\tau_2}\Theta_{c_2}B_{\Lambda_2}$  where   $\Lambda_1,\Lambda_2\subset\C^+$, $i=1,2$, are two Blaschke sequences. Then
\begin{enumerate} 
\item  $\K^2_{\Theta_1}\subseteq \K^2_{\Theta_2}$ if and only if $c_1\leq c_2$ and $\Lambda_1\subseteq \Lambda_2$.
\item  $\K^2_\Theta\subseteq \K^2_{\Theta_cB_\Lambda\Theta}$,  where  $c\geq 0$ and $\Lambda\subset\C^+$ a Blaschke sequence.
 \item If $0\leq c_1\leq c_2$, $\Lambda_1\subseteq \Lambda_2$, and $f\in\K^2_\Theta$, then 
$\Theta_{c_1}B_{\Lambda_1} f\in \K^2_{\Theta_{c_2}B_{\Lambda_2}\Theta}.$
\end{enumerate}
\end{lemma}

\begin{proof}
Let $\Theta_1,\Theta_2$ be inner functions. It is known that  $\K^2_{\Theta_1}\subseteq \K^2_{\Theta_2}$ if and only if $ {\Theta_2}/{\Theta_1}$ is an inner function.
Hence,  $\Theta_2/\Theta_1$ is inner if and only if 
$$
\frac{\Theta_2(z)}{\Theta_1(z)}=e^{i(\tau_2-\tau_1)}e^{i(c_2-c_1)z}\frac{B_{\Lambda_2}(z)}{B_{\Lambda_1}(z)}=\Theta(z),
$$
for some $\Theta$ inner. If $c_2\geq c_1$ and $\Lambda_1\subset\Lambda_2$, then $B_{\Lambda_2}/B_{\Lambda_1}=B_{\Lambda_2\backslash\Lambda_1}$ is a Blaschke product and $\Theta$ is therefore inner.

Now assume that $\Theta_2/\Theta_1$ is inner and that there exists $\lambda^\ast\in\Lambda_1$ that is not contained in $\Lambda_2$ (in the sense that for $\Lambda_i=\{\lambda_n^i\}_{n\in\N}$, $i\in\{1,2\}$, there is no injection $\phi:\N\rightarrow\N$ that satisfies $\lambda_n^1=\lambda_{\phi(n)}^2$ for every $n\in\N$), then 
$
\#\{n\in\N:\ \lambda^\ast=\lambda_n^2\}<\#\{n\in\N:\ \lambda^\ast=\lambda_n^1\},
$
which creates a pole at $\lambda^\ast$ for $\Theta_2/\Theta_1$, as the singular inner part is always nonzero, a contradiction to $\Theta_2/\Theta_1$ being inner. Then, as $e^{i(c_2-c_1)z}$ is bounded on $\C^+$ if and only if $c_2\geq c_1$, the first assertion follows. It only remains to show the last statement. By definition $f\in\K^2_\Theta$ if and only if $f\bot \Theta g$ for every $g\in\H^2$. Hence, as $\Theta_{c_1} B_{\Lambda_1}$ is an inner function we have
\begin{align*}
\langle \Theta_{c_1} B_{\Lambda_1} f,S_{c_1} B_{\Lambda_1}\Theta g\rangle&
=\int_\R \Theta_{c_1}(x)B_{\Lambda_1}(x) f(x)\overline{\Theta_{c_1}(x)B_{\Lambda_1}(x) \Theta(x)g(x)} \, \mbox{d}x\\
&=\int_\R  f(x)\overline{ \Theta(x)g(x)}\,\mbox{d}x=0,
\end{align*}
which shows $\Theta_{c_1}B_{\Lambda_1} f\in \K^2_{\Theta_{c_1}B_{\Lambda_1}\Theta}$. The result then follows from the first part of the corollary as $\Theta_{c_2}B_{\Lambda_2}\Theta=(\Theta_{c_2-c_1}B_{\Lambda_2\backslash\Lambda_1})(\Theta_{c_1}B_{\Lambda_1}\Theta)$.
\end{proof}

We are now ready to prove an oversampling theorem for model spaces.

\begin{theorem}\label{thm:over-sampling} Let $c>0$, and $\Theta$ be an inner function such that some  Clark measure for $\Theta_c \Theta$  is purely atomic and let $\{k^{\Theta_c\Theta}_{x_n}\}_{n\in \Z}$ be the corresponding basis  of $\K_{\Theta_c\Theta}^2$.
Then for every integer $m\geq 1$ and for every  $f\in\K^2_\Theta$, the following sampling formula holds: 
\begin{equation}
f(x)=\sum_{n\in\Z}f(x_n)e^{-\frac{ic(x-x_n)}{2}}\sinc\left(\frac{c(x-x_n)}{2  m}\right)^m \frac{ {k}^{\Theta_c\Theta}_{x_n}(x)}{\| {k}^{\Theta_c\Theta}_{x_n}\|^2},\qquad  x\in\R.
\end{equation}
{In particular, one can reconstruct  $f\in\K^2_\Theta$ from its samples  using an expansion with respect to functions of arbitrary polynomial decay.}

Further,
\begin{equation}
\label{eq:plancherelover}
\|f\|^2=\sum_{n\in\Z}\dfrac{|f(x_n)|^2}{\| {k}^{\Theta_c\Theta}_{x_n}\|^2},\qquad f\in \K_\Theta^2.
\end{equation}
\end{theorem}

\begin{proof}
If  $0\leq c_k \leq c/m$, $k=1,\ldots,m$, then $0\leq\sum_k c_k \leq c$, and 
$\Theta_{\sum_k c_k }f\in \K^2_{\Theta_c\Theta}$ by Lemma~\ref{cor:enlarge-theta}. Hence, using the sampling formula
\eqref{eq:clarkshan} for $\K^2_{\Theta_c\Theta}$, one has
$$
e^{i\sum_k c_k z}f(z)=\sum_{n\in\Z}f(x_n)e^{i\sum_k c_k  x_n}
\frac{ {k}^{\Theta_c\Theta}_{x_n}(z)}{\| {k}^{\Theta_c\Theta}_{x_n}\|^2}, \qquad z\in\C^+,
$$
and \eqref{eq:plancherelover} follows directly from \eqref{eq:plancherel}.
Further
\begin{align*}
f(z)&=\left(\frac{m}{c}\right)^m\int_0^{c/m}\ldots\int_0^{c/m} f(z)\,\mbox{d}c_1\cdots\mbox{d}c_m
\\
&=\left(\frac{m}{c}\right)^m\int_0^{c/m}\ldots\int_0^{c/m}\sum_{n\in\Z}f(x_n)e^{-i\sum_k c_k (z-x_n)}\frac{ {k}^{\Theta_c\Theta}_{x_n}(z)}{\| {k}^{\Theta_c\Theta}_{x_n}\|^2}\,\mbox{d}c_1\cdots\mbox{d}c_m
\\
&=\sum_{n\in\Z}f(x_n)\left(\frac{m}{ic(z-x_n)}\right)^m\left(1-e^{-\frac{ic(z-x_n)}{m}}\right)^m\frac{ {k}^{\Theta_c\Theta}_{x_n}(z)}{\| {k}^{\Theta_c\Theta}_{x_n}\|^2}.
\end{align*}
Moreover, we find that for almost every $x\in\R$,
\begin{align*}
f(x)&
=\sum_{n\in\Z}f(x_n)e^{-\frac{ic(x-x_n)}{2}}\left(\frac{m}{ic(x-x_n)}\left(e^{\frac{ic(x-x_n)}{2m}}-e^{-\frac{ic(x-x_n)}{2m}}\right)\right)^m \frac{ {k}^{\Theta_c\Theta}_{x_n}(x)}{\| {k}^{\Theta_c\Theta}_{x_n}\|^2}
\\
&=\sum_{n\in\Z}f(x_n)e^{-\frac{ic(x-x_n)}{2}}\sinc\left(\frac{c(x-x_n)}{2 m}\right)^m \frac{ {k}^{\Theta_c\Theta}_{x_n}(x)}{\| {k}^{\Theta_c\Theta}_{x_n}\|^2}.
\end{align*}
The oversampling formula is thus established.
\end{proof}
\bigskip


\section{Proof of Theorem \ref{th:intro2} based on oversampling}
\label{Sec:4}

We now have everything in place to prove Theorem \ref{th:intro2}.

Let $\Theta' \in L^\infty(\R)$. Then \eqref{eq:norm-Kt} and \eqref{eq:pointwise-Kx} imply
\begin{equation}\label{eq:golbal-Kx}
\frac{|k^\Theta_x(t)|}{\|k^\Theta_x\|}\leq \sqrt{\frac{\|\Theta'\|_{\infty}}{2\pi}},\qquad x,t\in\R.
\end{equation}
Let $ {k}^{\Theta_c\Theta}_z\in\K_{\Theta_c\Theta}^2$ be the reproducing kernel functions in $\K_{\Theta_c\Theta}^2$.
Denote by  $\{x_n\}_{n\in\Z}\subset\R$ the sequence associated in \eqref{eq:def-tn} with $\Theta_c\Theta$, that is,
$$
cx_n+\varphi(x_n)=\gamma+2\pi n, \qquad \gamma\in[0,2\pi), \ n\in\Z.
$$

For $c>0$ to be chosen later, set $\alpha=\dst \frac{c}{4}$ and $\Xi(x)=\sinc(x)$. 
First, from Theorem~\ref{thm:over-sampling}  (applied to  $m=2$) Fubini's theorem, and the bound \eqref{eq:golbal-Kx} we get  
\begin{align*}
\int_\R &|f(x)|^2\,\mbox{d}\mu(x)
=\int_\R\left|\sum_{n\in\Z}f(x_n)e^{-\frac{ic(x-x_n)}{2}}\sinc\left(\frac{c(x-x_n)}{4 }\right)^2
\frac{ {k}^{\Theta_c\Theta}_{x_n}(x)}{\| {k}^{\Theta_c\Theta}_{x_n}\|^2}\right|^2
\,\mbox{d}\mu(x)\\
&\leq \int_\R\sum_{n,k\in\Z}
\frac{|f(x_n)f(x_k)|}{\| {k}^{\Theta_c\Theta}_{x_n}\|\| {k}^{\Theta_c\Theta}_{x_k}\|}
\left|\Xi\bigl(\alpha(x-x_n)\bigr)\Xi\bigl(\alpha(x-x_k)\bigr)\right|^2
\frac{| {k}^{\Theta_c\Theta}_{x_n}(x) {k}^{\Theta_c\Theta}_{x_k}(x)|}{\| {k}^{\Theta_\Theta}_{x_n}\|\| {k}^{\Theta_c\Theta}_{x_k}\|}
\,\mbox{d}\mu(x)\\
&\leq\frac{c+ \|\Theta'\|_\infty}{2\pi}
\sum_{n,k\in\Z}\frac{|f(x_n)f(x_k)|}{\| {k}^{\Theta_c\Theta}_{x_n}\|\| {k}^{\Theta_c\Theta}_{x_k}\|} 
\int_\R \Xi\bigl(\alpha(x-x_n)\bigr)^2\Xi\bigl(\alpha(x-x_k)\bigr)^2\,\mbox{d}\mu(x).
\end{align*}
Applying the estimate \eqref{eq:schur} with $\eta\alpha$ (instead of $\delta$), $\eta$ to be chosen, we get
\begin{equation}
\label{eq:schur2}
\int_\R |f(x)|^2d\mu(x)
\leq4\frac{C_{\eta\alpha}^2}{\eta\alpha}(c+ \|\Theta'\|_\infty) \sum_{n,k\in\Z}\frac{|f(x_n)f(x_k)|}{\| {k}^{\Theta_c\Theta}_{x_n}\|\| {k}^{\Theta_c\Theta}_{x_k}\|} \frac{\sup_{x\in \R}\mu([x,x+\eta])}{4+\alpha^2(x_n-x_k)^2}
\end{equation}
with $C_{\eta\alpha}=\max(4,1+9\eta^2\alpha^2)$. We will choose $\eta\leq\delta$ with $\eta\alpha\leq 1/\sqrt{3}$,
{\it i.e.}, $\eta \leq\dfrac{4}{\sqrt{3}c}$, so that
$C_{\eta\alpha}=4$.
Now, by \eqref{eq:golbal-Kx},
$$
\begin{aligned}
2\pi |n-k| & =|cx_{n}+\varphi(x_{n})-cx_k-\varphi(x_k)| \\
& \leq (c+\|\varphi'\|_\infty) |x_{n}-x_k|
=(c+\|\Theta'\|_\infty) |x_{n}-x_k|.
\end{aligned}
$$
Set $u=(u_k)_k$ with
$u_k=\dfrac{|f(x_k)|}{\| {k}^{\Theta_c\Theta}_{x_k}\|}$ so that
$\norm{u}_{\ell^2(\Z)}^2=\norm{f}^2$ with \eqref{eq:plancherelover}.
Moreover, we set 
$$
\lambda=\dfrac{\pi^2\alpha^2}{(c+\|\Theta'\|_\infty)^2}
=\left(\dfrac{\pi c}{4(c+\|\Theta'\|_\infty)}\right)^2,
$$
$v=(v_k)_k$ with $v_k=\frac{1}{1+\lambda k^2}$ and notice that
$$
\norm{v}_{\ell^1(\Z)}= 1+2\sum_{k\geq 1}\frac{1}{1+\lambda k^2}
\leq 1+2\int_0^{\infty}\frac{\mbox{d}x}{1+\lambda x^2}
= 1+\frac{\pi}{\sqrt{\lambda}}
=1+\dfrac{4(c+\|\Theta'\|_\infty)}{c}.
$$
Then
\begin{align*}
\sum_{n,k\in\Z}\frac{|f(x_n)f(x_k)|}{\| {k}^{\Theta_c\Theta}_{x_n}\|\| {k}^{\Theta_c\Theta}_{x_k}\|} \frac{1}{4+\alpha^2(x_n-x_k)^2}
&\leq \frac{1}{4}\sum_{n\in\Z}u_n\sum_{k\in\Z}\frac{u_k}{1+\lambda(n-k)^2}\\
&=\frac{1}{4}\scal{u,u*v}\leq\frac{1}{4}\norm{u}_{\ell^2(\Z)}^2\norm{v}_{\ell^1(\Z)}.
\end{align*}
All in one, we get
\begin{align*}
\int_\R &|f(x)|^2\mbox{d}\mu(x)
\leq \frac{16}{\eta\alpha}M(\eta)\Big(c+\|\Theta'\|_\infty\Big)
\left(1+4 \frac{c+ \|\Theta'\|_\infty}{c}\right)
\|f\|^2,
\end{align*}
where $M(\eta):= \sup_{x\in \R}\mu([x,x+\eta])$. It remains to chose $c,\eta$. To do so, we distinguish two cases.

\noindent{\bf Case 1.} If $\|\Theta'\|_\infty\geq \dfrac{4}{\sqrt{3}\delta}$, we take $c=\|\Theta'\|_\infty$
and $\eta=\dfrac{4}{\sqrt{3}c}\leq \delta$ thus $\eta\alpha=\dfrac{1}{\sqrt{3}}$. Then
$$
\int_\R |f(x)|^2\mbox{d}\mu(x)
\leq 288\sqrt{3}\, M(\eta) \|\Theta'\|_\infty\|f\|^2
\leq 500\,   \Big(\frac{1}{\delta}+\|\Theta'\|_\infty\Big)M(\delta)\,\|f\|^2.
$$

\noindent{\bf Case 2.} If $\|\Theta'\|_\infty\leq \dfrac{4}{\sqrt{3}\delta}$,
we take $c=\dfrac{4}{\sqrt{3}\delta}$ and $\eta=\dfrac{4}{\sqrt{3}c}= \delta$
so that 
$$
\int_\R |f(x)|^2\mbox{d}\mu(x)
\leq 576\, \Big(\frac{1}{\delta}+\|\Theta'\|_\infty\Big)
M(\delta)\,\|f\|^2,
$$
as desired.\hfill $\Box$

\begin{remark}
In the proof, we have used the oversampling formula with $m=2$.
One may of course increase $m$ and this would lead to slightly better numerical
constants, at the price of much higher technicality. This requires in particular to replace
Lemma \ref{schur} with \eqref{eq:highertech}.
As the ultimate constants would anyway be worse than those we find in the next section,
we did not follow this path further.
\end{remark}
\bigskip


\section{An $L^p$-version of Theorem \ref{th:intro2} and Bernstein-type inequalities}
\label{Sec:5}

The previous proof gives a rather large constant $C$ in the inequality \eqref{1}. We present another proof 
that gives a better estimate;
in particular, we have the constant 1 in front of $\frac{1}{\delta}$, as in the classical Donoho--Logan theorem. Also, this proof
applies to $L^p$ analogs of the model spaces, the subspaces $\K_\Theta^p$ of
the Hardy space $\H^p$, where $1\le p<\infty$. 

Recall that for $1\le p\le\infty$ the subspace $\K_\Theta^p$ is defined as 
$$
\K_\Theta^p = \H^p \cap \Theta \overline{\H^p},
$$
where $\H^p$ is understood as a closed subspace of $L^p(\R)$.
This definition agrees with the one given for $\K^2_\Theta$ earlier. The properties of the spaces
$\K_\Theta^p$ are very much similar to the properties of $\K^2_\Theta$. In particular, if $\Theta$ 
is a meromorphic inner function, then all elements of $\K_\Theta^p$ are meromorphic functions in $\mathbb{C}$,
while for $\Theta(z) = e^{2icz}$ we have $\K_\Theta^p = e^{icz} PW_c^p(\mathbb{R})$,
where $PW_c^p(\mathbb{R})$ is the space of all entire functions of exponential type at most $c$
whose restriction to $\mathbb{R}$ belongs to $L^p(\mathbb{R})$.

\begin{theorem} 
\label{th:lp}
Let $1\leq p<\infty$, $\Theta' \in L^\infty(\R)$, $\mu$ be a Borel measure on $\R$, and let $\delta>0$
and $D_\mu(\delta)$ defined in \eqref{eq:dmudelta}.
Then for any $f\in \K_\Theta^p$ we have
\begin{equation}
\label{gh}
\|f\|_{L^p(\mu)}^p \le  \big(1+ \delta \|\Theta'\|_\infty\big)^pD_\mu(\delta)\|f\|_p^p.
\end{equation}  
\end{theorem}

The proof is based on a Bernstein-type inequality for model spaces 
$\K^p_\Theta$ with $\Theta'\in L^\infty(\R)$ which is due to Dyakonov \cite{dyak, dyak1}.

\begin{theorem}[Dyakonov]
If $\Theta'\in L^\infty(\R)$,  $1\leq p \leq \infty$, and $f\in \K_\Theta^p$, then 
$$\|f'\|_p \le \|\Theta'\|_\infty\|f\|_p.$$
\end{theorem}

Dyakonov proved this inequality up to some constant $C(p)$ on the right-hand side. Here we present a very simple proof
with the constant 1 which follows the method from \cite{B4} and is based on the formula  
\begin{equation}
\label{cont}
f'(x)= 2\pi i \int_\R f(t)\big(\overline{k_t^\Theta(x)}\big)^2\,\mbox{d}t,\qquad f\in \K_\Theta^p, \ x\in \R.
\end{equation}  

\begin{proof}[Proof of Dyakonov's theorem] 
For $1<p<\infty$, it follows from \eqref{cont}, \eqref{eq:norm-Kt} and the H\"older inequality that 
$$
\begin{aligned}
|f'(x)|^p & \le 2\pi \int_{\mathbb{R}} |f(t)|^p |k_x^\Theta(t)|^2\, \mbox{d}t \bigg(2\pi  \int |k_x^\Theta(t)|^2 \,\mbox{d}t \bigg)^{p/q} \\
& = 2\pi |\Theta'(x)|^{p/q} \int_{\mathbb{R}} |f(t)|^p |k_x^\Theta(t)|^2 \,\mbox{d}t \le 2\pi \|\Theta'\|_\infty^{p/q}
\int_{\mathbb{R}} |f(t)|^p |k_x^\Theta(t)|^2 \,\mbox{d}t.
\end{aligned}
$$
Since $\int_{\mathbb{R}} |k_x^\Theta(t)|^2 dx = |\Theta'(t)|/(2\pi)$, we get the result 
integrating over $x$. The cases $p=1$ or $p=\infty$ follow trivially from  \eqref{cont}.
\end{proof}

We will also need the following elementary lemma.

\begin{lemma} Let  $f,  f'\in L^p(\R)$, $1\le p<\infty$. Then for any 
$t_k \in [k\delta, (k+1)\delta] =:I_k$, $k\in{\mathbb Z}$, one has
\begin{equation}
\label{chev}
\bigg( \sum_{k\in{\mathbb Z}} |f(t_k)|^p\bigg)^{1/p} \le \delta^{-1/p} \|f\|_p + \delta^{1-1/p} \|f'\|_p.
\end{equation}
\end{lemma}

\begin{proof} From the triangle inequality, 
$$
\begin{aligned}
\bigg( \sum_{k\in{\mathbb Z}} & |f(t_k)|^p\bigg)^{1/p}  = 
\bigg( \frac{1}{\delta} \sum_{k\in{\mathbb Z}} \int_{I_k} |f(t_k)|^p \,\mbox{d}t \bigg)^{1/p} \\
& \le 
\bigg( \frac{1}{\delta} \sum_{k\in{\mathbb Z}} \int_{I_k} |f(t)|^p \,\mbox{d}t \bigg)^{1/p}
+
\bigg( \frac{1}{\delta} \sum_{k\in{\mathbb Z}} \int_{I_k} |f(t) - f(t_k)|^p \,\mbox{d}t\bigg)^{1/p}.
\end{aligned}
$$
Note that by H\"older's inequality
$$
\int_{I_k} |f(t_k) - f(t)|^p \,\mbox{d}t = \int_{I_k} \bigg|\int_t^{t_k} f'(s) \,\mbox{d}s \bigg|^p \,\mbox{d}t \le 
 \delta^{p/q} \int_{I_k} \int_t^{t_k} |f'(s)|^p \,\mbox{d}s \,\mbox{d}t \le \delta^p \int_{I_k} |f'(s)|^p \,\mbox{d}s.
$$
Summing these estimates we obtain \eqref{chev}.
\end{proof}

\begin{proof}[Proof of Theorem \ref{th:lp}]
Put $M(\delta) =  \sup_{x\in\R}\mu \big( [x, x+\delta) \big)=\delta D_\mu(\delta)$.
Let $t_k \in I_k$ be such that $|f(t_k)| = \sup_{I_k} |f|$. Then
$$
\int_{\R}|f(x)|^p\, \mbox{d}\mu(x) \le \sum_{k\in{\mathbb Z}} \mu\big([k\delta, (k+1)\delta)\big) \cdot |f(t_k)|^p  \le 
M(\delta) \big(\delta^{-1/p} \|f\|_p + \delta^{1-1/p} \|f'\|_p\big)^p.
$$
It follows from Dyakonov's theorem that
\begin{equation*}
\int_{\R}|f(x)|^p\, \mbox{d}\mu(x)  \le \frac{M(\delta)}{\delta} (1+ \delta \|\Theta'\|_\infty)^p \|f\|_p^p
\end{equation*}
as claimed.
%
\end{proof}

\begin{remark}\label{remark}
The dependence on the parameter $\delta$ and on $\|\Theta'\|_\infty$
in \eqref{gh} is sharp up to the numerical constant.
Indeed, for the classical Paley--Wiener space 
$PW_c^2(\R)$,  our method 
(applied to $PW_c^2(\R)$ in place of $\K^2_\Theta$ with $\Theta(z) = e^{2icz}$)
gives only a factor $(1+2\delta c)$, while the Donoho--Logan bound is $1+ \dfrac{c\delta}{\pi}$.

If  $\mu$ is given by $d\mu(x)=\mathbf{1}_T(x)dx$ and one considers the Paley--Wiener spaces,
it is common to choose $\delta$ to be the reciprocal of the bandwidth $2c$ and write the estimates 
in terms of the \emph{maximum Nyquist density} 
$$
D_{\max}(T,c):=\sup_{x\in\R}2c \cdot \big|T\cap [x,x+1/(2c)]\big| .
$$
Donoho and Logan's results then read s
$$
\|f\cdot \mathbf{1}_T\|_{p}^{p} \leq C_p \cdot D_{\max}(T,c)
\cdot \|f\|_p^p,
$$
for $p=1,2$, with constants $C_1=\sinc(1/4)^{-1}\approx 1.0105$ and $C_2=(1+1/2\pi)\approx 1.1592$,
while  our estimate \eqref{gh} gives $C_1=2$ and $C_2=4$.
\end{remark}
\bigskip


\section{A Donoho--Logan theorem for one component inner functions}
\label{Sec:6}

For general inner functions, even if we confine ourselves with meromorphic ones, 
one cannot expect to obtain estimates of Donoho--Logan  type considering intervals of a fixed length.
The reason is that the behaviour of inner functions on $\R$ can be very irregular. 
On the other hand, for meromorphic inner functions with a sublinear growth of the argument the assumption 
that the measure is uniformly bounded on intervals of a fixed length can be too strong.
It looks more natural then to consider the intervals where the change of the argument of $\Theta$ is
fixed, {\it i.e.}, consider the intervals $I=[a,b]$ such that $\phi(b) - \phi(a) = \delta$, where $\phi$ is an increasing continuous
argument for $\Theta $ on $\R$. There exists a class of inner functions for which such a generalization is possible. 
These are the so-called {\it one-component inner functions}, that is, those inner functions
for which the sublevel set 
$$
\Omega(\Theta, \vep) = \{z\in \mathbb{C}^+:  |\Theta(z)| <\vep\}
$$
is connected for some $\vep \in (0,1)$. One-component inner functions were introduced by 
W.\,S.~Cohn \cite{cohn} in connection to Carleson-type embeddings of model spaces 
and were subsequently studied by many authors (see, {\it e.g.}, \cite{Al1, Al2, B4, B5, nr}). 

Several important properties of one-component inner functions were obtained by
A.\,B.~Aleksandrov \cite{Al1}. Recall that the spectrum $\rho(\Theta)$ of $\Theta$ (see \eqref{spec})
is a closed set, and if we write $\R \setminus \rho(\Theta) = 
\cup_n J_n$, where $J_n$ are disjoint open intervals, then $\Theta$ admits an analytic continuation through 
any $J_n$. Hence, $\Theta$ has an increasing $C^\infty$ branch of the argument $\phi$ on each $J_n$.
It is shown in \cite{Al1} that  $\rho(\Theta)$ has zero Lebesgue measure and, if we additionally assume that
$\infty\in\rho (\Theta)$, the function $\dfrac{1}{\phi'} = \dfrac{1}{|\Theta'|}$ (defined as 0 on $\rho(\Theta)$)
is a Lipschitz function on $\R$. Moreover, it is shown in \cite{B4} that in this case for any $\vep\in (0,1)$ 
there exist positive constants $c_1,c_2$ depending on $\vep$ only such that
\begin{equation}
\label{dist}
c_1 |\Theta'(x)|^{-1} \le {\rm dist}\, (x, \Omega(\Theta, \vep)) \le c_2 |\Theta'(x)|^{-1}, \qquad x\in\R.
\end{equation}
Therefore, there exist $A, B>0$  such that
\begin{equation}
\label{dc}
A \le \frac{\phi'(s)}{\phi'(t)} \le B, \qquad |\phi(s)-\phi(t)| \le 1.
\end{equation}
These properties essentially characterize the class of one-component inner functions. 

Also in  \cite{B4} the following Bernstein type inequality was proved:
if $\Theta$ is a one-component inner function and  $1<p<\infty$, then 
\begin{equation}
\label{dc1}
\|f'/\Theta'\|_p \le C(\Theta, p) \|f\|_p, \qquad f\in K_\Theta^p.
\end{equation}

The next theorem applies to general one-component (not necessarily meromorphic) functions. 

\begin{theorem}
\label{th:onecomp}
Let $\Theta$ be a one-component inner function such that $\infty\in\rho(\Theta)$, $1<p<\infty$.
Let $\phi$ denote the branch of the argument of $\Theta$ which is an increasing $C^\infty$ function on
each subinterval of $\R \setminus \rho(\Theta)$. For a Borel measure $\mu$ on $\R$ and $\delta>0$ put 
$$
D_\mu^\Theta(\delta) = \sup\Big\{ \frac{\mu(I)}{|I|}: \ I=[a,b], \ \phi(b) - \phi(a) =\delta \Big\}.
$$
Then for any $f\in \K_\Theta^p$ we have
\begin{equation}
\label{gh2}
\int_\R |f|^p\, \mbox{d}\mu \le D_\mu^\Theta(\delta) (1+C\delta)^p\|f\|_p^p,
\end{equation}  
where the constant $C>0$ depends on $\Theta $ and $p$ only.
\end{theorem}

\begin{proof}
The proof is analogous to the proof of Theorem \ref{th:lp}. 
It follows from \eqref{dist} that $\phi$ is unbounded on any connected component of
$\R \setminus \rho(\Theta)$. Therefore, we can
write $\R \setminus \rho(\Theta) = \cup_k I_k$,  where $I_k=[a_k, b_k]$ are intervals with disjoint interiors
such that $\phi(b_k) - \phi(a_k) =\delta$. Let $D=D_\mu^\Theta(\delta)$.
Any function $f\in  \K_\Theta^p$ admits an analytic continuation through $\R \setminus \rho(\Theta)$.
Choose $t_k\in I_k$ such that $|f(t_k)| = \max_{I_k} |f|$. Then we have 
$$
\begin{aligned}
\|f\|_{L^p(\mu)} & \le \bigg( \sum_k |f(t_k)|^p \mu(I_k)\bigg)^{1/p} \le D^{1/p}
\bigg( \sum_k  \int_{I_k} |f(t_k)|^p \,\mbox{d}t \bigg)^{1/p} \\
& \le M^{1/p} \bigg( \sum_k  \int_{I_k} |f(t)|^p \,\mbox{d}t \bigg)^{1/p} +
D^{1/p} \bigg( \sum_k  \int_{I_k} |f(t) -f(t_k)|^p \,\mbox{d}t \bigg)^{1/p}.
\end{aligned}
$$
Furthermore, if $p'$ is the dual index to $p$, $\dfrac{1}{p}+\dfrac{1}{p'}=1$, then
$$
\begin{aligned}
\sum_k  \int_{I_k} |f(t) -f(t_k)|^p \,\mbox{d}t & =
\sum_k  \int_{I_k} \bigg| \int_{t_k}^t f'(s) \,\mbox{d}s \bigg|^p \,\mbox{d}t  \\
& \le \sum_k  \int_{I_k} \int_{t_k}^t \frac{|f'(s)|^p}{|\Theta'(s)|^p} \,\mbox{d}s
\bigg( \int_{t_k}^t  |\Theta'(s)|^{p'} \,\mbox{d}s\bigg)^{p/p'} \,\mbox{d}t.
\end{aligned}
$$ 
For any $k$ there exists $s_k\in I_k$ such that $\delta = \phi(b_k) - \phi(a_k)  = \phi'(s_k) \cdot |I_k|$.
In view of \eqref{dc} we have $|\Theta'(s)| \cdot|I_k| \le C_1 \delta$, $s\in I_k$, whence
$$
\bigg( \int_{t_k}^t  |\Theta'(s)|^q \,\mbox{d}s\bigg)^{p/p'}
\le C_2 \bigg( \frac{\delta^{p'}}{|I_k|^{p'-1}}\bigg)^{p/p'}  = 
C_2\frac{\delta^p}{|I_k|}.
$$
for some constants $C_1, C_2>0$. Thus,
\begin{align*}
\sum_k  \int_{I_k} \int_{t_k}^t \frac{|f'(s)|^p}{|\Theta'(s)|^p} \,\mbox{d}s
\bigg( \int_{t_k}^t  |\Theta'(s)|^{p'} \,\mbox{d}s\bigg)^{p/p'} \,\mbox{d}t &\le C_2\frac{\delta^p}{|I_k|} 
\sum_k  \int_{I_k} \int_{I_k} \frac{|f'(s)|^p}{|\Theta'(s)|^p} \,\mbox{d}s \,\mbox{d}t \\ &\le C_3 \delta^p \|f\|^p_p
\end{align*}
by \eqref{dc1}. We have shown that
$$
\|f\|_{L^p(\mu)} \le D^{1/p} (1+ C_4\delta)\|f\|_p,
$$
which proves the theorem.
\end{proof}

\subsection*{Acknowledgements}  M.S. would like to acknowledge the support of the Austrian Science Fund FWF through the projects J-4254, and Y-1199.

The work of A.B. 
was supported by 
the Russian Science Foundation project 19-71-30002.

The research of K.K. is supported by the project ANR-18-CE40-0035.

Ce travail a bénéficié d'une aide de l'\'Etat attribu\'e \`a l'Universit\'e de Bordeaux en tant qu'Initiative d'excellence, au titre du plan France 2030.

\bibliographystyle{plain}

\end{document}